\documentclass[]{amsart}

\usepackage{amscd,amsthm,amssymb,amsfonts,amsmath,euscript}

\theoremstyle{plain}
\newtheorem{thm}{Theorem}[section]
\newtheorem{lemma}[thm]{Lemma}
\newtheorem{prop}[thm]{Proposition}
\newtheorem{cor}[thm]{Corollary}

\theoremstyle{definition}
\newtheorem{defn}[thm]{Definition}
\newtheorem{eg}[thm]{Example}

\theoremstyle{remark}
\newtheorem{remark}[thm]{Remark}
\newtheorem*{thank}{{\bf Acknowledgments}}
\newcommand{\nc}{\newcommand}

\def\makeop#1{\expandafter\def\csname#1\endcsname
  {\mathop{\rm #1}\nolimits}\ignorespaces}
\makeop{Hom}   \makeop{End}   \makeop{Aut}   \makeop{Isom}  \makeop{Pic} 
\makeop{Gal}   \makeop{ord}   \makeop{Char}  \makeop{Div}   \makeop{Lie} 
\makeop{PGL}   \makeop{Corr}  \makeop{PSL}   \makeop{sgn}   \makeop{Spf}
\makeop{Spec}  \makeop{Tr}    \makeop{Nr}    \makeop{Fr}    \makeop{disc}
\makeop{Proj}  \makeop{supp}  \makeop{ker}   \makeop{im}    \makeop{dom}
\makeop{coker} \makeop{Stab}  \makeop{SO}    \makeop{SL}    \makeop{SL}
\makeop{Cl}    \makeop{cond}  \makeop{Br}    \makeop{inv}   \makeop{rank}
\makeop{id}    \makeop{Fil}   \makeop{Frac}  \makeop{GL}    \makeop{SU}
\makeop{Trd}   \makeop{Sp}    \makeop{Tr}    \makeop{Trd}   \makeop{diag}
\makeop{Res}   \makeop{ind}   \makeop{depth} \makeop{Tr}    \makeop{st}
\makeop{Ad}    \makeop{Int}   \makeop{tr}    \makeop{Sym}   \makeop{can}
\makeop{length}\makeop{SO}    \makeop{torsion} \makeop{GSp} \makeop{Ker}
\makeop{Adm}
\def\makebb#1{\expandafter\def
  \csname bb#1\endcsname{{\mathbb{#1}}}\ignorespaces}
\def\makebf#1{\expandafter\def\csname bf#1\endcsname{{\bf
      #1}}\ignorespaces} 
\def\makegr#1{\expandafter\def
  \csname gr#1\endcsname{{\mathfrak{#1}}}\ignorespaces}
\def\makescr#1{\expandafter\def
  \csname scr#1\endcsname{{\EuScript{#1}}}\ignorespaces}
\def\makecal#1{\expandafter\def\csname cal#1\endcsname{{\mathcal
      #1}}\ignorespaces} 

\def\doLetters#1{#1A #1B #1C #1D #1E #1F #1G #1H #1I #1J #1K #1L #1M
                 #1N #1O #1P #1Q #1R #1S #1T #1U #1V #1W #1X #1Y #1Z}
\def\doletters#1{#1a #1b #1c #1d #1e #1f #1g #1h #1i #1j #1k #1l #1m
                 #1n #1o #1p #1q #1r #1s #1t #1u #1v #1w #1x #1y #1z}
\doLetters\makebb   \doLetters\makecal  \doLetters\makebf
\doLetters\makescr 
\doletters\makebf   \doLetters\makegr   \doletters\makegr
     \def\qed{\qedmark\medbreak}%
\def\qedmark{{\enspace\vrule height 6pt width 5pt depth 1.5pt}}%
    \def\setminus{\smallsetminus}

\normalsize

\makeop{Bl}

\def\Spec{{\rm Spec}\,}

\newcommand{\Z}{\mathbb Z}
\newcommand{\Q}{\mathbb Q}
\newcommand{\R}{\mathbb R}
\newcommand{\C}{\mathbb C}

\nc{\embed}{\hookrightarrow}

\newcommand{\ch}{characteristic }

\newcommand{\dieu}{Dieudonn\'{e} }

\nc{\ol}{\overline}
\nc{\wt}{\widetilde}
\nc{\wh}{\widehat}
\nc{\opp}{\mathrm{opp}}

\def\Mat{{\rm Mat}}

\makeop{Ram}
\makeop{Rep}

\begin{document}
\renewcommand{\thefootnote}{\fnsymbol{footnote}}
\setcounter{footnote}{-1}
\numberwithin{equation}{section}


\title{On the existence of maximal orders}
\author{Chia-Fu Yu}
\address{
Institute of Mathematics, Academia Sinica \\
6th Floor, Astronomy Mathematics Building \\
No. 1, Roosevelt Rd. Sec. 4 \\ 
Taipei, Taiwan, 10617 and NCTS(Taipei Office)} 
\email{chiafu@math.sinica.edu.tw}

\date{\today}

\begin{abstract}
We generalize the existence of maximal orders in a semi-simple algebra
for general ground rings. We also improve several statements in Chapter 5
and 6 of Reiner's book \cite{reiner:mo} 
concerning separable algebras by removing the
separability condition, provided the ground ring is only assumed to be
Japanese, a very mild condition. 
Finally, we show the existence of maximal orders
as endomorphism rings of abelian varieties in each isogeny class.
\end{abstract} 

\maketitle


\def\char{{\rm char\,}}

\section{Introduction}
\label{sec:01}

Maximal orders are basic objects in the integral theory of
semi-simple algebras. As a generalization of the rings of integers in
number fields, they are also main interests in number theory. 
A classical result states that the existence of maximal orders, not
just for the ring of integers in a number field, may hold
in a quite general setting, which we describe now (Theorem~\ref{11}). 

Let $R$ be a Noetherian integral domain with quotient field $K$.  
We only consider $K$-algebras which are finite-dimensional. 
A (finite-dimensional) $K$-algebra $A$ is said to be {\it separable 
  over $K$} if
it is semi-simple and the center $Z(A)$ of $A$ is a separable
(commutative) semi-simple $K$-algebra, that is, $Z(A)$ is a finite
product of finite separable field extensions of $K$. 
Clearly any central simple
$K$-algebra is separable. 
For a $K$-algebra $A$, an {\it $R$-order} $\Lambda$ in
$A$ is a finite $R$-subring of $A$ which spans $A$ over $K$. An
$R$-order $\Lambda$ of $A$ is said to be {\it maximal} if there is no
$R$-order 
$\Lambda'$ of $A$ which strictly contains $\Lambda$. 
The following is a classical result about the existence of maximal
orders; see \cite[Corollary 10.4 and Theorem 10.5 (iv),
p.~127--8]{reiner:mo}. 

\begin{thm}\label{11}
  Let $R$ and $K$ be as above, and $A$ a
  semi-simple algebra over $K$. Then there exists a maximal $R$-order
  of $A$ provided one of the following conditions hold:
  \begin{enumerate}
  \item $R$ is normal and $A$ is separable over $K$. 
  \item $R$ is a complete discrete valuation ring.
  \end{enumerate}
\end{thm}

When $\char K=0$ or even $K$ is a number field, Theorem~\ref{11} provides
most of the situations we may encounter. However, when $K$ is a global
function field, the assumption of the separability of $A$ 
seems to be superfluous. We would like to find a necessary and
sufficient condition for the ground ring $R$ so that maximal orders in
any semi-simple $K$ algebra exists. 
In this Note we prove the following result, 
which removes the separability assumption in Theorem~\ref{11} (1) 
for rather general ground rings in positive characteristics.

\begin{thm}\label{12}
  Let $R$ be a Noetherian integral domain and $K$ be its quotient field. 
\begin{enumerate}
  \item Assume that $R$ is a Japanese ring. Then any $R$-order of a
    semi-simple $K$-algebra $A$ is contained in a maximal $R$-order. In
    particular, every semi-simple $K$-algebra contains a maximal
    $R$-order. 
  \item Conversely, if 
  every semi-simple $K$-algebra $A$ contains a maximal 
  $R$-order, then $R$ is a Japanese ring.
\end{enumerate}
\end{thm}

We shall recall the definition of Japanese and Nagata rings as well as
some of their properties and the relationship to 
(quasi-)excellent rings. 
Nagata domains are special cases of Japanese rings. 
Examples of Nagata rings include commutative rings of finite type over
$\Z$
and their localizations, commutative rings of finite type over
any field $k$ and their localizations, and Noetherian complete 
semi-local rings. 
Noetherian normal domains with quotient field of \ch zero are Japanese
rings. 

In the second part of this Note we give a description of maximal
orders in a semi-simple $K$-algebra, where the ground ring $R$ is
either a Noetherian Japanese ring or an excellent ring. We reduce the
description to the case when $R$ is a Noetherian normal domain, whose
description becomes well-known. For the convenience of the reader, 
we also include  
the expository account of this important theory. 
Our reference is the well-written book by I.~Reiner \cite{reiner:mo}. 

Note that the results of this Note generalize all statements
concerning separable $K$-algebras $A$ in Chapters 5 and 6 of Reiner's
book \cite{reiner:mo}. We remove the separability condition for $A$ 
provided the Dedekind domain $R$ is assumed to be either excellent or
Japanese; see the reduction step in Section~\ref{sec:03} or
Subsection~\ref{sec:41}. For number theorists, this assumption is
harmless. 

Our motivation of entering the integral theory of semi-simple algebras
is due to the basic fact that the endomorphism ring of an abelian
variety is an order of a semi-simple $\Q$-algebra. An abelian variety
whose endomorphism ring is maximal should be distinguished from others
in its isogeny class. The last part of this Note shows that in any isogeny
class of abelian varieties there is an abelian variety whose
endomorphism ring is maximal, 
a result about the existence of maximal orders. 
More precisely, we show the following result.


\begin{thm}\label{ab}
  Let $A_0$ be an abelian variety over an arbitrary field $k$. Let 
  $\calO' \subset \End^0_k(A_0)=\End_k(A_0)\otimes_\Z \Q$ be a maximal
  order containing $\End_k(A_0)$. Then there exist an abelian variety
  $A'$ over $k$ and an isogeny $\varphi:A_0 \to A'$ over $k$ such that
  with the identification $\End^0_k(A_0)=\End^0_k(A')$ by $\varphi$
  one has $\End_{k}(A')=\calO'$. Moreover, the
  isogeny $\varphi$ can be chosen to be minimal with respect to
  $\calO'$ in the following sense: if $(A_1, \varphi_1)$ is
  another pair such that $\End_{k}(A_1)=\calO'$, then there is an 
  (necessarily unique) isogeny $ \alpha: A' \to A_1$
  such that $\varphi_{1}=\alpha\circ \varphi$. 
\end{thm}

A local version of Theorem~\ref{ab} (where $\calO'$ is a maximal
order of $\End(A_0)\otimes \Q_p$ containing $\End(A_0)\otimes \Z_p$) 
when $k$ is an algebraically
closed field of \ch $p>0$ is used in the proof of the reduction step 
\cite[Lemma 2.4]{yu:endo}. \\ 

This Note is organized as follows. In Section~\ref{sec:02} we review
the properties of Japanese, Nagata, and excellent rings, and their 
relationship as well. We also discuss the relationship between the 
properties (regularity, normality and some others) of a local ring
and its completion 
In Section~\ref{sec:03} we show the existence of maximal orders
in a semi-simple algebra with Japanese ground rings. In
Section~\ref{sec:04} we attempted to describe maximal orders in these
semi-simple algebras $A$ and show that the description can be reduced to
the case where the ground ring is a complete discrete valuation ring
and $A$ may be assumed to be central simple. 
We collect the description of maximal orders in a central simple
algebra over a discrete valuation ring and a Dedekind domain,
following Reiner \cite{reiner:mo}. 
In the last section we give the proof of Theorem~\ref{ab}.

\section{Nagata and excellent rings}
\label{sec:02}

In this section we recall the definition of Nagata rings and excellent
rings, as well as their properties. Our references are Matsumura
\cite{matsumura:ca80}, and EGA IV \cite{ega4:20,ega4:24}. 

Notations here are independent of Section 1, as we prefer to follow
closely Matsumura \cite{matsumura:ca80} and EGA IV. 
All rings and algebras in this section are commutative with identity.


\subsection{Nagata rings}
\label{sec:21}

\begin{defn}\ Let $A$ be an integral domain with quotient field $K$.
\begin{enumerate}
  \item  We say that {\it $A$ is N-1} if the integral closure $A'$ of $A$ in
    its quotient field $K$ is a finite $A$-module. 
  \item We say that {\it $A$ is N-2} if for any finite extension field $L$
    over $K$, the integral closure $A_L$ of $A$ in $L$ is a finite
    $A$-module.   
\end{enumerate}
\end{defn}

If $A$ is N-1 (resp. N-2), then so is any localization of $A$. 
If $A$ is a Noetherian domain of \ch zero, then $A$ is N-2 if and only
if $A$ is N-1. This follows immediately from the basic theorem that 
if $L$ is a finite {\it separable} field extension of $K$ 
and $A$ is a Noetherian {\it normal} domain, 
then the integral closure $A_L$ of $A$ in $L$ is
finite over $A$ (cf. \cite[Proposition 31.B, p.~232]{matsumura:ca80}). 

\begin{defn}
  A ring $A$ is said to be {\it Nagata} if 
\begin{enumerate}
  \item $A$ is Noetherian, and  
  \item $A/\grp$ is N-2 for any prime ideal $\grp$ of $A$. 
\end{enumerate}
\end{defn}

If $A$ is Nagata, then any localization of $A$ and any finite
$A$-algebra are all Nagata. Nagata uses the term
``pseudo-geometric rings'' for such rings as coordinate rings of
varieties over any field all share this property. 
Nagata rings are the same
as what are called Noetherian universal Japanese rings in EGA IV
\cite[23.1.1, p.~213]{ega4:20}, which we recall now.  

\begin{defn}\
\begin{enumerate}
\item An integral domain $A$ is said to be {\it Japanese} if it is N-2.
\item An ring $A$ is said to be {\it universal Japanese} if any finitely
  generated integral domain over $A$ is Japanese.  
\end{enumerate}
\end{defn}

From the definition, a universal Japanese ring is not required to be
an integral domain nor to be Noetherian. A universal Japanese domain is
Japanese. It follows 
from the definition that any Noetherian universal Japanese ring is
Nagata. Conversely, the following theorem 
\cite[Theorem 72, p.~240]{matsumura:ca80}, due to Nagata, shows that 
any Nagata ring is also a Noetherian universal Japanese ring.

\begin{thm}\label{nagata} 
  If $A$ is a Nagata ring, then so is any finitely generated $A$-algebra.
\end{thm}

The proof of this theorem is quite involved and the reader to
referred to Matsumura \cite[\S\, 31]{matsumura:ca80}. The following
provides some more examples (see \cite[Corollaries 1 and 2,
p.~234]{matsumura:ca80}).

\begin{prop}\
\begin{enumerate}
\item If $A$ is a Noetherian normal domain which is N-2, then the
  formal power series ring $A[[X_1,\dots, X_n]]$ is N-2 also.
\item Any Noetherian complete local ring $A$ is a Nagata ring. 
\end{enumerate}
\end{prop}

\def\nor{{\rm Nor}}

For any scheme $X$, let $\nor (X)$ denote the subset of $X$ that
consists of normal points. 

\begin{lemma} Let $A$ be a Noetherian domain and $X:=\Spec A$.
  \begin{enumerate}
  \item If there is a non-zero element $f\in A$ such that
    $A_f:=A[1/f]$ is normal, then $\nor(X)$ is open in $X$.
  \item If $A$ is N-1, then $\nor(X)$ is open in $X$.   
  \end{enumerate}
\end{lemma}
\def\HT{{\rm ht}}
\begin{proof}
  (1) This is Lemma 3 of Matsumura \cite[\S\,31.G,
      p.~238]{matsumura:ca80} and its proof is sketched there. 
      We provide more
      details for the convenience of the reader. 
      Using a criterion for normality \cite[Theorem 39,
      p.~125]{matsumura:ca80},
      for $\grq\in \Spec A$, the integral domain $A_\grq$ is normal if
      and only if it satisfies the conditions ($R_1$) and ($S_2$),
      that is, $A_\grp$ is regular for all prime ideals $\grp\subseteq
      \grq$ with 
      $\HT(\grp)=1$, and ${\rm Ass}(A_\grq/f')$, the set of associated
      prime ideals of $A_\grq/f'$ for all $0\neq f'\in \grq$, has no
      embedded prime ideals 
      (cf. \cite[p.~125]{matsumura:ca80}). Let
\[ E:=\{\grp\in {\rm Ass}(A/f)\,|\, \HT(\grp)=1 \text{\ and\ } A_\grp
      \text{\ is not regular, or\  } \HT(\grp)>1\, \}. \]  
      Clearly $E$ is a finite subset.     
      We claim that \[ \nor(X)=X-\bigcup_{\grp\in E} V(\grp). \]
      Let $\grq$ be a prime ideal not contained in $\cup_{\grp\in E}
      V(\grp)$. We shall show that $\grq$ is a normal point. 
      If $\grq\in \Spec A[1/f]$, then $\grq$ is a normal
      point by our assumption. Suppose that $f\in \grq$. If $\grp\in
      {\rm Ass}(A_\grq/f)$, then $\HT(\grp)=1$. This means ${\rm
      Ass}(A_\grq/f)$ has no embedded prime ideals and hence $A_\grq$
      satisfies ($S_2$). On the other hand, let $\grp\subseteq
      \grq$ be prime ideals with $\HT(\grp)=1$. If $f\not\in \grp$,
      then $\grp$ is a normal point and $(A_\grq)_\grp=A_\grp$ is
      regular. If $f\in \grp$, then $(A_\grq)_\grp=A_\grp$ is
      regular, by the definition of $E$.. 
      This shows that $\grq\in \nor(X)$ and the proof of (1) is
      completed.  
   
  (2) Let $A'$ be the normalization of $A$,
      and let $X':=\Spec A'$. Since $A$ is N-1, the natural morphism
      $X'\to X$ is a finite dominant birational morphism. Then there
      is a non-zero element $f\in A$ such that the restriction to the
      open subset 
      \[ X'_f:=\Spec A'[1/f]\to X_f:=\Spec A_f \]
is an isomorphism. In particular, $A_f$ is normal. It follows from (1)
that $\nor(X)$ is open in $X$. 

We provide another simpler proof of (2), which is not based on
(1). Put $M:=A'/A$; this is a finite $A$-module as $A$ is N-1. 
For each $\grp\in
X$, we have
\[ M_\grp=(A')_\grp/A_\grp=(A_\grp)'/A_\grp, \]
as the operations localization and normalization commute. It follows that
\[ \nor(X)=\{\grp\in X\, |\, M_\grp=0\, \}. \]
Since $A$ is Noetherian and $M$ is finite over $A$, $\nor(X)$ is open
in $X$.  \qed  
\end{proof}

Let $A$ be a Noetherian semi-local ring and $A^*$ its completion. If
$A^*$ is reduced, then $A$ is said to be {\it analytically
  reduced}. 

\begin{thm}\cite[Theorem 70, p.~236]{matsumura:ca80}\label{an_unram}
  Any Nagata semi-local domain is analytically reduced. 
\end{thm}

This is useful for checking non-Nagata rings; see
Example~\ref{non-excellent}. \\
 

\subsection{G-rings and closedness of singular loci}
\label{sec:32}

Recall that a Noetherian local ring $(A,\grm,k)$ is said to be a 
{\it  regular local ring} if $\dim A=\dim_k \grm/\grm^2$
\cite[p.~78]{matsumura:ca80}. 
A Noetherian ring $A$ is said to be {\it regular} if all local rings
$A_\grp$ are regular local for $\grp\in\Spec A$. One can  show
that if the local ring $A_\grm$ is regular for all maximal ideals
$\grm$ of $A$, then $A$ is regular 
\cite[\S\,18.G Corollary, p.~139]{matsumura:ca80}. Any regular local
ring is an integral domain. 

\begin{defn}(\cite[\S\,33, p.~249]{matsumura:ca80}). 
\begin{enumerate}
  \item Let $A$ be a Noetherian ring containing a field $k$. We say
    that $A$ is {\it geometrically regular over $k$} if for any finite
    field extension $k'$ over $k$, the ring $A\otimes_k k'$ is 
    regular. This is equivalent to saying that the local ring $A_\grm$
    has the same property for all maximal ideals $\grm$ of $A$.
  \item Let $\phi:A\to B$ be a homomorphism (not necessarily of finite
    type) of Noetherian rings. We say that $\phi$ is {\it regular} if it is
    flat and for
    each $\grp\in \Spec A$, the fiber ring $B\otimes_A k(\grp)$ is
    geometrically regular over the residue field $k(\grp)$. 
  \item A Noetherian ring $A$ is said to be a {\it G-ring} if for each $\grp
    \in \Spec A$, the natural map 
    $\phi_\grp: A_\grp\to (A_\grp)^*$ is regular,
    where $(A_\grp)^*$ denotes the completion of the local ring
    $A_\grp$. 
\end{enumerate}
\end{defn}

Note that the natural map $\phi_\grp: A_\grp\to (A_\grp)^*$ is faithfully
flat. The fibers of the natural morphism $\Spec (A_\grp)^*\to \Spec
A_\grp$ are called formal fibers. To say a Noetherian ring $A$ is a
$G$-ring then is equivalent to saying that all formal fibers of the
canonical map $\phi_\grp$ for each prime ideal $\grp$ of $A$ 
are geometrically regular. 
It is clear that, if $A$ is a G-ring, then any localization $S^{-1}A$
of $A$ and any homomorphism image $A/I$ of $A$ are G-rings.  


\begin{thm}\cite[Theorem 93, p.~279]{matsumura:ca80}.
  Let $(A,\grm$ be a Noetherian local ring containing a field
  $k$. Then $A$ is geometrically regular over $k$ if and only if $A$
  is formally smooth over $k$ in the $\grm$-adic topology. 
\end{thm}

\begin{lemma}\cite[Lemma 2, p.~251]{matsumura:ca80}.
  let $\phi:A\to B$ be a faithfully flat, regular homomorphism. Then
  \begin{enumerate}
  \item $A$ is regular (resp. normal, resp. Cohen-Macaulay,
    resp. reduced) if and only if $B$ has the same property;
  \item If $B$ is a G-ring, then so is $A$.
  \end{enumerate}
\end{lemma}

\def\reg{{\rm Reg}}

For a Noetherian scheme $X$, let $\reg (X)$ denote the subset of $X$
that consists of regular points.
 
\begin{defn} Let $A$ be a Noetherian ring. 
  \begin{enumerate}
  \item We say that $A$ is {\it J-0} if $\reg (\Spec A)$ contains a
    non-empty open set of $\Spec A$.
  \item We say that $A$ is {\it J-1} if $\reg (\Spec A)$ is open in  
    $\Spec A$.
  \end{enumerate}
\end{defn}

If $A$ is a integral domain, then $\reg(\Spec A)$ is non-empty and
hence the condition J-1 implies J-0. Indeed, the the generic point of
$\Spec A$ is a regular point as the localization of $A$ is its
quotient field, which is a regular local ring.

\begin{thm}\label{j2} \cite[Theorem 73, p.~246]{matsumura:ca80}.
  For a Noetherian ring $A$, the following conditions are equivalent
  \begin{enumerate}
  \item any finitely generated $A$-algebra $B$ is J-1;
  \item any finite $A$-algebra $B$ is J-1;
  \item for any $\grp\in \Spec A$, and for any finite radical
    extension $K'$ of $k(\grp)$, there exists a finite $A$-algebra
    $A'$ satisfying $A/\grp\subseteq A'\subseteq K'$ which is J-0 and
    whose quotient field is $K'$. 
  \end{enumerate}
\end{thm}

\begin{defn} A Noetherian ring $A$ is {\it J-2} if it satisfies one
  of the equivalent conditions in Theorem~\ref{j2}.
\end{defn}

\begin{remark}
  The condition (3) of Theorem~\ref{j2} is satisfied if $A$ is a
  Nagata ring of dimension one. Indeed, $A/\grp$ is either a field or a
  Nagata domain of dimension one. In the first case, (3) is
  trivial. In the second case, the integral closure $A'$ of $A$ in
  $K'$ is finite over $A$ and is a regular ring. Therefore, any
  Nagata ring of dimension one is a J-2. On the other hand, we have
  Theorem~\ref{g2} (2). 
\end{remark}

We gather some properties of G-rings.

\begin{thm}\label{g1}\
  \begin{enumerate}
  \item Any complete Noetherian local ring is a G-ring. 
  \item If for any maximal ideal $\grm$ of a Noetherian ring $A$, the
    natural map $A_\grm\to (A_\grm)^*$ is regular, then $A$ is a
    G-ring
  \item Let $A$ and $B$ be Noetherian rings, and let $\phi:A\to B$ be
    a faithfully flat and regular homomorphism. If $B$ is J-1, then
    so is $A$.
  \item A semi-local G-ring is J-1.    
  \end{enumerate}
\end{thm}
\begin{proof}
  (1) See \cite[Theorem 68, p.~225 and p.~250]{matsumura:ca80}. (2)
      See \cite[Theorem 75, p.~251]{matsumura:ca80}. (3) and (4) See 
      \cite[Theorem 76, p.~252]{matsumura:ca80}. 
\end{proof}

\begin{thm}\label{g2}\
  \begin{enumerate}
  \item Let $A$ be a G-ring and $B$ a finitely generated
    $A$-algebra. Then $B$ is a G-ring.
  \item Let $A$ be a G-ring which is J-2. Then $A$ is a Nagata ring.
  \end{enumerate}
\end{thm}
\begin{proof}
  (1) See \cite[Theorem 77, p.~254]{matsumura:ca80}. 
  (2) See \cite[Theorem 78, p.~257]{matsumura:ca80}. 
\end{proof}

\begin{thm}{\rm (Analytic normality of normal G-rings).} 
  Let $A$ be a G-ring and $I$ an ideal of $A$. Let $B$ be the $I$-adic
  completion of $A$. Then the canonical map $A\to B$ is
  regular. Consequently, if $B$ is normal (resp. regular,
  resp. Cohen-Macaulay, resp. reduced) so is $A$.
\end{thm}
\begin{proof}
  See \cite[Theorem 79, p.~258]{matsumura:ca80}.
\end{proof}

\subsection{Excellent rings}
\label{sec:33}

\begin{defn} \cite[\S\, 34, p.~259]{matsumura:ca80}.
Let $A$ be a Noetherian ring.
  \begin{enumerate}
  \item We say that $A$ is {\it quasi-excellent} if the following
    conditions are satisfied:
    \begin{itemize}
    \item [(i)] $A$ is a G-ring;
    \item [(ii)] $A$ is J-2.
    \end{itemize}
  \item We say that $A$ is {\it excellent} if it satisfies (i), (ii)
    and the following condition
    \begin{itemize}
    \item [(iii)] $A$ is universally catenary.
    \end{itemize}
  \end{enumerate}
\end{defn}

We recall the following \cite[p.~84]{matsumura:ca80}:

\begin{defn}\
  \begin{enumerate}
  \item A ring $A$ is said to be {\it catenary} if for any two prime
    ideals $\grp\subseteq \grq$, the relative height ${\rm ht}(\grq/\grp)$
    is finite and is
    equal to the length of any maximal chain of prime ideals between
    them.
  \item A Noetherian ring $A$ is said to be {\it universally catenary} if
    any finitely generated $A$-algebra is catenary.
  \end{enumerate}
\end{defn}

\begin{remark}\
\begin{enumerate}
\item Each of the conditions (i), (ii), and (iii) is stable under the
localization and passage to a finitely generated algebra
(Theorems \ref{j2} (1) and \ref{g2} (1)).
\item Note that (i), (ii), (iii) are conditions on $A/\grp$, $\grp\in
  \Spec A$. Thus a Noetherian ring $A$ is (quasi-)excellent if and
  only if $A_{\rm red}$ is so.
\item The conditions (i) and (iii) are of local nature (in the sense
  that if they hold for $A_\grp$ for all $\grp\in \Spec A$, then they
  hold for $A$), while (ii) is not. 
\item  Theorem~\ref{g2} (2) states that any quasi-excellent ring is a
  Nagata ring. 
\item It follows from Theorems~\ref{j2} and \ref{g1} (4) that any
  Noetherian local G-ring is quasi-excellent. 
\end{enumerate}
\end{remark}

\begin{eg}\label{excellent} \cite[\S\,34.B, p.~260]{matsumura:ca80}.  
  \begin{enumerate}
  \item Any complete Noetherian semi-local ring is excellent.
  \item Convergent power series rings over $\R$ or $\C$ are excellent.
  \item Any Dedekind domain $A$ of \ch zero is excellent. 

\end{enumerate}
\end{eg}

\begin{eg} \label{non-excellent} (cf. \cite[\S\,34.B,
  p.~260]{matsumura:ca80}).  

There exists a regular local domain of dimension one, that is, a
    discrete valuation ring, in \ch
    $p>0$ which is not excellent. Take a field $k$ of \ch $p>0$ with
    $[k:k^p]=\infty$. Put $R:=k[[t]]$ and let $A$ be the subring of $R$
    consisting of power series 
\[ \sum_{n=0}^\infty a_n t^n, \quad \text{with}\quad
[k^p(a_0,a_1,\dots): k^p]<\infty. \]
Then $A$ is a regular local ring of dimension one with uniformizer $t$
and the completion $A^*$ is equal to $R$. 
Let $K$ be the quotient field of $A$. 
The formal fiber of the natural map $A\to A^*=R$ at the generic point
is given by $K\to R[1/t]=k((t))$. Since $R^p\subset A$, 
the quotient field $k((t))$ of
$R$ is purely inseparable over $K$. Note that a field extension $K'$ over
a field $k'$ is geometrically regular if and only if $K'/k'$ is
separable. Therefore, $k((t))$ is not
geometrically regular over $K$. This shows that $A$ is not a G-ring.  

We show that $A$ is not a Nagata ring either. Suppose on the contrary 
that $A$ is a Nagata ring. Choose an element $c\in R - A$.  
Then $b:=c^p$ lies in $A$ and
the ring $A[c]$ is finite over $A$ and hence a Noetherian semi-local
ring. By Theorem~\ref{nagata}, $A[c]$ is again a Nagata ring. By
Theorem~\ref{an_unram}, the completion $A[c]^*$ should be reduced. However,
the completion
\[ A[c]^*=A^*\otimes_A A[c]=A^*[t]/(t^p-b)=A^*[t]/(t-c)^p \]
is not reduced, contradiction. 

Note that the ring $A[c]$ is not a
discrete valuation ring as it is not integrally closed. If let $L$ be the
quotient field of $A[c]$ and $B$ the integral closure of $A$ in
$L$, then $B$ is not a finite $A$-module. Indeed, write 
\[ c=a_0+a_1 t+ a_2 t^2+\cdots \]
and define, for $i\ge 0$, 
\[ c_i:=a_i+a_{i+1} t +a_{i+2} t^2+\cdots. \]
Then 
\[ c_i=a_i+t c_{i+1}, \quad \forall\, i\ge 0, \]
and each element $c_i$ is contained in $L$ and integral over $A$. 
We now have an increasing sequence of subrings in $L$ 
\[ A[c]=A[c_0]\subseteq A[c_0,c_1]=A[c_1]\subseteq
A[c_0,c_1,c_2]=A[c_2]\subseteq \cdots \]
which is not stationary as $c\not\in A$. This shows that $B$ is not a
finite $A$-module, and hence $A$ is not Japanese. \\    
\end{eg}

 
\begin{remark}
  We explain that if there exists a non-Japanese discrete valuation
  ring $A$, then $A$ is essentially of the form as in
  Example~\ref{non-excellent}. Let $k$ be the residue field of $A$,
  then the 
  completion $A^*$ of $A$, by Cohen's structure theorem for complete
  regular local rings \cite[Corollary 2, p.~205]{matsumura:ca80}, 
  is isomorphic to $k[[t]]=:R$ and one has $\char k=p>0$ (otherwise
  $A$ is an excellent ring; see Example~\ref{excellent} (3)). Since
  the natural map $A\to A^*$ is faithfully flat, we may regard $A$ as
  a dense subring of $R$, and may also choose $t$ as a 
  uniformizing element of $A$. Since the quotient
  field $k((t))$ should be inseparable over $K$, it is natural to
  expect that $A$ contains the subring $R^p=k^p[[t^p]]$. Since $A$
  contains $k$ and $t$, $A$ should contain the image $C$ of
  $k\otimes_{k^p} k^p[[t]]$ in $R$. If $[k:k^p]<\infty$, then
  $R=C\subseteq A$ and
  there is no such an example. So we have to assume
  $[k:k^p]=\infty$. Under this assumption, the image $C$ of
  $k\otimes_{k^p} k^p[[t]]$ in $R$ is exactly the example constructed 
  in Example~\ref{non-excellent}.   
\end{remark}

\subsection{Relation with the completion}
\label{sec:24}
Let $A$ be a local ring. We have
\[ \text{regular}\implies \text{normal}\implies \text{integral}
\implies \text{reduced.}
\]  
We discuss the relationship of $A$ with its completion $A^*$ through
these properties.

\begin{prop} \label{ff} 
Let $f:A \to B$ be a flat local homomorphism of local rings. 
  \begin{enumerate}
  \item If $B$ is reduced (resp. integral and integrally closed), then
    so is $A$. 
  \item If $B$ is a regular local ring, then so is $A$.
  \end{enumerate}
\end{prop}
\begin{proof}
  (1) This is elementary; we keep a proof simply for the convenience
      of the reader. 
      Since $f$ is faithfully flat, we may regard $A$ as subring of
      $B$. Thus, $A$ is reduced or integral if $B$ is so. Assume $B$ is
      a normal domain. Let $L$ (resp. $K$) be the quotient field of
      $B$ (resp. $A$), and we have $K\subset L$. 
      Since $B$ is integrally closed, the integral
      closure $A'$ of $K$ is equal to $B\cap K$. Since $B$ has no
      non-zero $A$-torsion element, the natural map  $K\otimes_A B\to
      KB\subset L$ is injective. This shows
      that the map $A\otimes_A B\to A'\otimes_A B\cong B$ is an
      isomorphism. By faithful flatness, we get $A=A'$. 

\def\pd{{\rm proj.dim\, }}  
\def\gd{{\rm gl.dim\,}}
\def\Tor{{\rm Tor}}
  (2) We learned the proof from C.-L.~Chai. First, we show $A$ is
      Noetherian. Let $I_1\subset I_2\subset \cdots$ be an increasing
      sequence of ideals of $A$. Then $I_i\otimes_A B\cong I_iB$,
      $i=1,\dots,$ form an increasing sequence of ideals of $B$. As $B$
      is Noetherian and $B$ is faithfully flat over $A$, the (ACC)
      holds for $A$. 

      To show the regularity of $A$, we recall the following
      definitions and results. The {\it projective dimension} of a
      module $M$ over a ring $A$, denoted by $\pd M$, is 
      defined to be the length of a shortest
      projective resolution of $M$. 
      The {\it global dimension} of $A$, denoted by $\gd A$, 
      is defined to be 
      \[ \gd A:={\rm supp}_M\, \{\pd M \}, \]
      where $M$ runs through all
      $A$-modules, or equivalently all finite $A$-modules \cite[Lemma
      2, p.~128]{matsumura:ca80}. When $A$ is
      Noetherian, we have 
\[ \gd A\le n \iff
      \Tor_{n+1}^A(M,N)=0 \]
       for all finite $A$-modules $M$ and $N$ 
      \cite[Lemma 5, p.~130]{matsumura:ca80}.   
      A theorem of Serre states that a Noetherian local ring $A$ is
      regular if and only if the global dimension of $A$ is finite
      (cf. \cite[Theorem 45, p.~139]{matsumura:ca80}). Now we are ready
      to prove the regularity. Since $B$ is flat over $A$, the functor
      $B\otimes_A$ commutes with the Tor functors, and hence we have
      the canonical isomorphism
\[ B\otimes_A \Tor^A_{i} (M,N)\cong \Tor^B_{i} (B\otimes_A
      M,B\otimes_A N) \]
for any $A$-modules $M$ and $N$. Since $B$ has finite global dimension,
the latter vanishes for all $M$ and $N$ when $i>\gd B$. By faithful
flatness, we have $\Tor^A_{i} (M,N)=0$ for $i>\gd B$ and hence $A$
has finite global dimension. By Serre's theorem, $A$ is regular. \qed
\end{proof}

\begin{cor}
  Let $f:Y\to X$ be a flat morphism of schemes. If $y$ is $f$ a point of
  of $Y$ and $O_y$ is reduced (resp. normal, resp. regular), then
  so is $O_{f(y)}$.
\end{cor}

\begin{lemma}
  Let $A$ be a local ring. Then $A$ is regular if and only if its
  completion $A^*$ is so.
\end{lemma}
\begin{proof}
  The implication $\implies$ follows from 
\[ \dim A^*=\dim A=\dim \grm_A/\grm_A^2=\dim \grm_{A^*}/\grm_{A^*}^2,
\]
where $\grm_A$ (resp. $\grm_{A^*}$) is the maximal ideal of $A$
  (resp. $A^*$). The other implication follows from
  Proposition~\ref{ff} (2).\qed
\end{proof}

We now have the implications
\[ A^* \text{ is P} \implies A \text{ is P}. \]
where P is normal, integral, or reduced, and
\[ A^* \text{ is regular} \iff A \text{ is regular}. \]
without any condition for the local ring $A$. 

It is also well-known that the implication 
\[ A \text{ is an integral domain} \implies A^* \text{ is an integral
  domain} \]
is wrong even when $A$ is excellent. For example, take
$A=k[x,y]_{(x,y)}/(y^2-x^2-x^3)$, where $k$ is any field of \ch $p\neq
2$.  

When P is reduced or normal, the implications
\[ A^* \text{ is P} \implies A \text{ is P}, \]
need some conditions on $A$, for example, the morphism 
$\varphi: A\to A^*$ is reduced (resp. normal), i.e. it is flat and 
the fibers are geometrically reduced (resp. geometrically normal). 
As we have the implications
\[ \text{regular} \implies \text{normal} \implies \text{reduced} \]
for morphisms (this follows from the definition), 
we have the implications
\[ A \text{ is reduced (resp. normal)} \implies A^* \text{ is reduced
  (resp. normal)} \]
when $A$ is a G-ring.

Recall that a local ring $A$ is said to be {\it unibranched} if its
reduced ring $A_{\rm red}$ is an integral domain, and the
normalization $A'_{\rm red}$ of $A_{\rm red}$ is again local (EGA IV
\cite[23.2.1, p.~217]{ega4:20}). It is clear that a normal local domain
is unibranched. We know that the completion of a Noetherian local
normal domain may not be reduced. Is the reduced ring of its
completion still an integral domain, or even unibranched? We find in
\cite[E7.1, p.~210]{nagata:local} that Nagata constructed a
Noetherian local normal domain $A$ such that (1) its completion $A^*$
is reduced, and (2) $A^*$ is not an integral domain. Therefore, the
completion $A^*$ is not unibranched.



\section{Existence of maximal orders}
\label{sec:03}
In this section we give the proof of Theorem~\ref{12}. 

Let $R$ be a
Noetherian integral domain and $K$ its quotient field. Assume that $R$
is a Japanese ring. Let $A$ be a semi-simple algebra over $K$ and
$\Lambda$ be an $R$-order of $A$. Let $Z$ be the center of $A$ and
write
\begin{equation}
  \label{eq:31}
  Z=\prod_{i=1}^r Z_i 
\end{equation}
as a product of finite field extensions of $K$. This gives rise to a
decomposition of the semi-simple algebra 
\begin{equation}
  \label{eq:32}
  A=\prod_{i=1}^r A_i
\end{equation}
into simple factors, and each simple factor $A_i$ is central simple
over $Z_i$. Let $R'$ be the
integral closure of $R$ in $Z$. Then 
\begin{equation}
  \label{eq:33}
  R'=\prod_{i=1}^r R'_i,
\end{equation}
where $R'_i$ is the integral closure of $R$ in $Z_i$ for each $i$. 
Choose a system of generators $x_1,\dots, x_m$ of $\Lambda$ over
$R$ (as $R$-modules). Let $\Lambda'$ be the $R'$-submodule of $A$
generated by these 
$x_j$'s; clearly $\Lambda'$ is an $R'$-subring. 
Since $R$ is Japanese, the ring $R'$ is a finite
$R$-module. Then any $R'$-subring of $A$ is an $R'$-order if and only
if it is an $R$-order, in particular, $\Lambda'$ is an $R$-order 
containing $\Lambda$. The decomposition (\ref{eq:33}) gives rise to
the decomposition
\begin{equation}
  \label{eq:34}
  \Lambda'=\prod_{i=1}^r \Lambda'_i
\end{equation}
and each factor $\Lambda'_i$ is an $R'_i$-order in $A_i$. Since $A_i$ is
central simple over $Z_i$ and $R'_i$ is a Noetherian normal domain, by
Theorem~\ref{11} (cf. \cite[Corollary 10.4]{reiner:mo}) there exists a
maximal $R_i'$-order $\Lambda''_i$ of $A_i$ containing $\Lambda'_i$
for each $i$. Then the product 
\[ \Lambda'':=\prod_{i=1}^r \Lambda''_i \]
is a maximal $R'$-order and hence $R$-order of $A$ containing
$\Lambda$. This completes the proof of Theorem~\ref{12} (1). 

We show the second statement. Let $A$ be a finite field extension of
$K$. Any $R$-order in $A$ is contained in the
integral closure $R_A$ of $R$ in $A$. Let $\Lambda$ be a maximal
$R$-order in $A$. Then one has $\Lambda=R_A$, otherwise 
one can make a bigger
$R$-order $\Lambda[c]$ by adding an element $c\in R_A\setminus
\Lambda$.
This shows that the integral closure $R_A$ is the unique  maximal
$R$-order in $A$. Thus, $R_A$ is a finite $R$-module. Therefore, the
ring $R$ is Japanese. This
completes the proof of Theorem~\ref{12}.    
 

\section{Description of maximal orders}
\label{sec:04}
Keep the notation of Section~\ref{sec:01}.
In this section we attempted to describe maximal $R$-orders in $A$
when $R$ is either a Noetherian Japanese ring or an excellent ring.

\subsection{Reduction to normal domains}
\label{sec:41}

Let $\Lambda$ be an $R$-order of a semi-simple $K$-algebra $A$, where
$R$ is a Noetherian Japanese ring or an excellent domain. Let 
\[ Z=\prod_{i=1}^r Z_i, \quad A=\prod_{i=1}^r A_i, \quad \text{and\ }
R'=\prod_{i=1}^r R'_i \]
be as in Section~\ref{sec:03}. If $\Lambda$ is a maximal $R$-order,
then $\Lambda$ contains the subring $R'$, that gives rise to the
decomposition 
\[ \Lambda=\prod_{i=1}^r \Lambda_i, \]
and each factor $\Lambda_i$ is a maximal $R'_i$-order of
$A_i$. Conversely, if we are given a maximal $R'_i$-order $\Lambda_i$
of $A_i$ for each $i$, then the product $\prod_{i=1}^r \Lambda_i$ is a
maximal $R'$-order of $A$, and is also a maximal $R$-order of $A$ as
$R'$ is finite over $R$. 

Therefore, the description of maximal $R$-orders in $A$ can be reduced
to the case where $A$ is central simple over $K$ and $R$ is a Noetherian
normal domain, or an excellent normal domain if the initial ground ring
is excellent.

\subsection{Noetherian normal domain cases}
\label{sec:42}
Let $R$ be a Noetherian normal domain with quotient field $K$ and $A$
a central simple algebra over $K$. We recall the following results.

\begin{prop}
  An $R$-order $\Lambda$ in $A$ is maximal if and only if for each
  maximal ideal $\grm$ of $R$, the localization $\Lambda_\grm$ is a
  maximal $R_\grm$-order in $A$. 
\end{prop}
\begin{proof}
  See \cite[Corollary 11.2, p.~132]{reiner:mo}. 
\end{proof}\

We say that an $R$-order $\Lambda$ in $A$ is {\it reflexive} if the
inclusion $\Lambda\subseteq \Lambda^{**}$ is an equality, where 
\[ \Lambda^*=\Hom_R(\Lambda,R), \quad
\Lambda^{**}=\Hom_R(\Lambda^*,R). \]

\begin{thm}[Auslander-Goldman]\label{ag}
  An $R$-order $\Lambda$ is maximal if and only if $\Lambda$ is
  reflexive and for each minimal non-zero prime $\grp$ of $R$, the localization
  $\Lambda_\grp$ is a maximal $R_\grp$-order.
\end{thm}
\begin{proof}
  See \cite[Theorem 11.4, p.~133]{reiner:mo}. 
\end{proof}\

Using Theorem~\ref{ag}, we may even reduce the description of maximal orders
to the case where $R$ is a discrete valuation ring. By the following
theorem, we may even pass to their completions.  

\begin{thm}\label{43}
  Assume that $R$ is an excellent local normal domain. Let $\wh R$ be the
  completion of $R$ and $\wh K$ the quotient field of $\wh R$. Let
  $\Lambda$ be $R$-order in $A$, and set
\[ \wh \Lambda:=\wh R \otimes_R \Lambda, \quad 
\wh A:=\wh K \otimes_K A,  \]
so $\wh \Lambda$ is an $\wh R$-order in $\wh A$. Then $\Lambda$ is a
maximal $R$-order in $A$ if and only if $\wh \Lambda$ is a maximal
$\wh R$-order in $\wh A$. 
\end{thm}
\begin{proof}
  This is \cite[Theorem 11.5, p.~133]{reiner:mo}. 
\end{proof}\


Note that the assumption of excellence for $R$ is not stated in
 \cite[Theorem 11.5]{reiner:mo}. 
That would cause a problem as the completion of $R$ may not 
be a domain (see Nagata \cite[Appendix]{nagata:local} for the examples). 
For the special case where $R$ is a
discrete valuation ring, the original statement holds, 
as the completion $\wh R$ is again a discrete valuation ring. See also
Subsection~\ref{sec:24} for more details about the relationship of a
Noetherian local ring  with its completion. 






\subsection{Complete discrete valuation ring cases}
\label{sec:43}
 Let $R$ be a complete discrete valuation ring with the unique maximal
 ideal $P=\pi R \neq 0,$ $K$ its quotient field, and
 $\overline{R}=R/P.$ Let $A$ be a central simple $K$-algebra, and $V$
 be a minimal left 
 ideal of $A.$ Set $D:=\Hom_A(V,V).$
 Then $D$ is a division algebra, by Schur's Lemma, whose center is
 equal to $K$. The minimal left ideal $V$ naturally forms a 
right $D$-vector space, and one has 
$A=\Mat_r(D)$, where $r:=\dim_D V$. Let $v$ be the normalized
 $P$-adic valuation on 
 $K$, that is, $v(\pi)=1$. Let $N_{D/K}$ be the reduced norm
 on $D$ and define 
\[ w(a):=[D:K]^{-1/2} v(N_{D/K}(a)),  \ a \in D. \]


It is easy to see the following (see \cite[Theorem 12.8,
  p.~137]{reiner:mo}, \cite[Theorem 12.10,
  p.~138]{reiner:mo} and \cite[Theorem 13.2, 
  p.~139]{reiner:mo}).

\begin{lemma} \label{44}\
\begin{enumerate}
\item The valuation $w$ is the unique extension of $v$ to $D$ and the
  ring of integers 
$$ \Delta =\{a
\in D : \ w(a)\geq 0\} $$
is the unique maximal $R$-order in $D$.

\item Let $\pi_D$ be a prime element of $\Delta,$ and set $\wp
    =\pi_D \Delta.$ Then every non-zero one-sided ideal of $\Delta$ is
    a 
 two-sided ideal, and is a power of $\wp.$ The residue class ring
 $\ol \Delta:=\Delta/\wp$ is again a division algebra over the field
 $\overline{R}$,  and $\wp \cap R=P.$ 
\end{enumerate}
\end{lemma}


 

Lemma~\ref{44} describes the maximal orders $\Lambda$ in $D$ (in fact
$\Lambda$ is unique) and all
ideals of $\Lambda$. Now we look at the central simple case.

\begin{thm}\label{45} \
\begin{enumerate}
\item Let $\Lambda =\Mat_r(\Delta).$ Then $\Lambda$ is a maximal
  $R$-order in $A,$ and has a unique maximal two-sided ideal $\pi_D 
\Lambda.$ The powers $$(\pi_D \Lambda)^m=\pi_D^m \Lambda, \ \ \
m=0,1,2,\ldots,$$ give all of the nonzero two-sided ideals of 
$\Lambda.$
 \item Every maximal $R$-order in $A$ is of the form $u \Lambda
   u^{-1}$ for some unit $u \in A,$ and each such ring is a maximal 
$R$-order.
 \item Every maximal order $\Lambda^{\prime}$ is left and right
   hereditary, and each of its one-sided ideals is principal. The 
unique maximal two-sided ideal of $u \Lambda u^{-1}$ is $u\cdot \pi_D
\Lambda \cdot u^{-1}.$ 
\item Let $\Lambda$ be any maximal $R$-order in $A$. Then there exists
  a full free $\Delta$-lattice $M$ in $V$ such that
  $\Lambda=\Hom_{\Delta}(M,M).$ Conversely, each such $\Lambda$ is
  maximal. 
\end{enumerate}
\end{thm}

\begin{proof}
  See \cite[Theorem 17.3, p.~171]{reiner:mo} and \cite[Corollary 17.4,
  p.~172]{reiner:mo}.  \\
\end{proof}

See Subsection~\ref{sec:51} for the definition of hereditary rings.
Note that in Theorem~\ref{45} (4) any $\Delta$-lattice $M$ in $V$
  is free automatically, and hence any two full $\Delta$-lattices in
  $V$ are isomorphic. Therefore, the statements (2) and (4) in
  Theorem~\ref{45} are equivalent. For more general ground rings, 
  the analogue of (4) is weaker than that of (2) in general; see
  also Theorem~\ref{56}.
   
\subsection{Discrete valuation rings cases}
\label{sec:44}
\def\rad{{\rm rad}\,}

 Keep the notations as in
 Subsection~\ref{sec:43}, but the ground ring $R$ now is only assumed to
 be a discrete valuation ring. Let $\wh R$, $\wh K$, $\wh A$ be the
 same as in Theorem~\ref{43}.  
 If $\Lambda$ is an $R$-order in $A,$ then
 $\widehat{\Lambda}:=\widehat{R}\otimes_R \Lambda$ is again 
 $\widehat{R}$-order $\widehat{\Lambda}$ in $\widehat{A}$, and it is
 maximal if and only if so is $\Lambda$.


\begin{thm}
Let $\Lambda$ be a maximal $R$-order in a central simple $K$-algebra
$A.$ Then $\Lambda$ has a unique maximal two-sided ideal 
$\mathfrak{P},$ given by $\mathfrak{P}= \Lambda \cap \rad 
\widehat{\Lambda}.$ Then $\rad \Lambda = \mathfrak{P},$ and
every nonzero two-sided ideal of 
$\Lambda$ is a power of $\mathfrak{P}.$ Further, $\rad
\widehat{\Lambda}$ is the $P$-adic completion of $\rad
\Lambda.$ 
\end{thm}

\begin{proof}
  See \cite[Theorem 18.3, p.~176]{reiner:mo}.\\
\end{proof}

The following provides a criterion for an $R$-order to be maximal.

\begin{thm}[Auslander-Goldman]
 Let $\Lambda$ be an $R$-order in the central simple $K$-algebra $A.$
 Then $\Lambda$ is maximal if and only if $\Lambda$ is 
 hereditary, and $\rad\Lambda$ is its unique maximal two-sided ideal.
\end{thm}
\begin{proof}
  See \cite[Theorem 18.4, p.~176]{reiner:mo}.
\end{proof}

\begin{thm}\label{48}
Let $\Lambda$ be a maximal $R$-order in a central simple $K$-algebra $A.$
\begin{enumerate}
\item The ring $\Lambda$ is left and right hereditary.
\item Let $M$ and $N$ be left $\Lambda$-lattices. Then $M \cong N$ if
  and only if $M$ and $N$ have the same rank. 
\item Every one-sided ideal of $\Lambda$ is principal.
\item Every maximal $R$-order in $A$ is a conjugate $u\Lambda u^{-1}$
  of $\Lambda,$ where $u$ is a unit of $A.$ 
\item Let $\widehat{A}=\widehat{K}\otimes A \cong \Mat_t(E),$ where $E$
  is a division algebra with center $\widehat{K},$ and let $\Omega$ 
  be the unique maximal $\widehat{R}$-order in $E.$ Then
  $$\Lambda/\rad\Lambda \cong \Mat_t(\Omega/\rad\Omega)$$ and 
  $\Omega/\rad\Omega$ is a division algebra.
\end{enumerate}
\end{thm}

\begin{proof}
  See \cite[Theorem 18.1, p.~175]{reiner:mo} and 
  \cite[Theorem 18.7, p.~179]{reiner:mo}.\\
\end{proof}

We see that the description of maximal orders $\Lambda$ and that of
all ideals of $\Lambda$ are similar to the case where $R$ is
complete. However, the maximal orders $\Delta$ in $D$, the division
part of $A$, may not be unique, as the valuation $v$ may not be
extended to $D$ {\it uniquely}. It is the case exactly when the 
completion $\wh D:=D\otimes \wh K$ remains a division algebra. 

\section{Maximal orders over Dedekind domains}
\label{sec:05}

In this section, we give the expository description of maximal $R$-orders
in a central simple algebra, where $R$ is a Dedekind domain 
Our reference is again I.~Reiner \cite{reiner:mo}.

\subsection{Hereditary rings} 
\label{sec:51}
We recall
\begin{defn}\cite[\S\, 2f, p.~27]{reiner:mo}
  A (not necessarily commutative) ring $\Lambda$ with identity 
  is said to be {\it left
  hereditary} (resp. {\it right hereditary}) if every left
  (resp. right) ideal of $\Lambda$ is a projective
  $\Lambda$-module. 
\end{defn}

\begin{lemma}\label{submod}
  Let $\Lambda$ be a left hereditary ring, and $N$ a
  $\Lambda$-submodule of a finite free left $\Lambda$-module. Then $N$
  is isomorphic to an external finite direct sum of left ideals of
  $\Lambda$, and is therefore projective. 
\end{lemma}
\begin{proof}
  One can show this easily by induction; see \cite[Theorem 2.44,
  p.~28]{reiner:mo}. \qed   
\end{proof}

\begin{remark}\

\begin{enumerate}
  \item If the ring $\Lambda$ has the property that every submodule of
    finite free left modules are projective, then in particular all left
    ideals of $\Lambda$ are projective. Therefore, by
    Lemma~\ref{submod}, $\Lambda$ is left hereditary if and only if
    submodules of finite free left modules are projective. 

  \item There are examples of rings which are left hereditary but not
    right hereditary. However, if $\Lambda$ is left and right
    Noetherian, then $\Lambda$ is left hereditary if and only if
    $\Lambda$ is right hereditary (this fact is due to Auslander,
    cf. \cite[p.~29]{reiner:mo}). Therefore, we may simply say
    $\Lambda$ hereditary in this case. 
  \item Lemma~\ref{submod} also holds without the finiteness for the free
    module. One can use transfinite induction to prove this.    
\end{enumerate}
\end{remark}

\begin{thm}[Steinitz] {\rm (}cf. \cite[(4.1), p.~45 and Theorem 4.13,
    p.~49]{reiner:mo}{\rm )} \label{structure}
  \begin{enumerate}
  \item Every Dedekind domain $R$ is hereditary. Therefore, 
  every finitely
  generated $R$-module $M$ without torsion elements 
  is isomorphic to an external finite direct sum 
\[ M\simeq J_1\oplus \dots, \oplus J_n, \]
  where $\{J_i\}$ are ideals of $R$, and $n=\rank_R M:=\dim_K M\otimes_R
  K$. In particular, $M$ is a projective $R$-module. 
  \item Two such sums $\oplus_{i=1}^n  J_i$ and $\oplus_{i=1}^m  J'_i$
    are $R$-isomorphic if and only if $m=n$, and the products
    $J_1\cdots J_n$ and $J_1'\cdots J'_m$ are in the same ideal class. 
  \end{enumerate}
\end{thm}

See \cite[Lemma 3]{reiner:pams57} for a simple proof of
Theorem~\ref{structure} (2). 

\begin{thm}\label{55}
  Let $R$ be a Dedekind domain with quotient field $K$, and $A$ a
  central simple $K$-algebra. Let $\Lambda$ be an $R$-order in
  $A$. Then $\Lambda$ is hereditary if and only if the localization
  $\Lambda_\grp$ is hereditary for every prime ideal $\grp$ of $R$.   
\end{thm}
\begin{proof}
  See \cite[Theorem 40.5, p.~368]{reiner:mo}. 
\end{proof}

\subsection{Maximal orders over Dedekind domains} 
\label{sec:52}
Let $R$ be a Dedekind domain with quotient field $K$, and assume $K\neq R$.
Let $A$ be a central simple algebra over $K$. We may and do identify $A$
as $\Hom_D(V,V)$, where $D$ is a central simple division algebra over
$K$, and $V$ is a finite-dimensional right vector space over $D$. 
Choose a maximal $R$-order $\Delta$ of $D$; it exists by
Theorem~\ref{11} or \ref{12} though not necessarily unique. 

\begin{thm}\label{56}
  Notation as above. If $M$ is a full right $\Delta$-lattice in $V,$
  then $\Lambda:=\Hom_{\Delta}(M,M)$ is a maximal $R$-order in
  $A$. Conversely, 
  for any maximal $R$-order $\Lambda'$ in $A$, there exists a full
  right $\Delta$-lattice $N$ in $V$ such that
  $\Lambda'=\Hom_{\Delta}(N,N)$. 
\end{thm}
\begin{proof}
  See \cite[Theorem 21.6, p.~189]{reiner:mo}. 
\end{proof}\

If $M$ and $N$ are $\Delta$-isomorphic, then there is an element $g\in
A$ such that $N=gM$. In this case, $\Lambda'=g \Lambda
g^{-1}$. Conversely, if $\Lambda'$ is conjugate to $\Lambda$ by an
element in $A$, then any $\Delta$-module $N$ with
$\Lambda'=\Hom_{\Delta}(N,N)$ is isomorphic to $M$ as
$\Delta$-modules. 
In general, the set of conjugacy classes of maximal $R$-orders may not
be singleton; its cardinality, if is finite, is called the {\it type
  number} of $A$. 

The description of maximal orders (Theorem~\ref{56}) is generalized by 
Auslander and Goldman
\cite{auslander-goldman:mo} to the case where $R$
is a regular domain but $A$ is a matrix algebra over $K$. They show
that any maximal order $\Lambda$ in $A=\End_K(V,V)$, where $V$ is a
finite-dimensional $K$-vector space, is of the form $\Hom_R(M,M)$,
where $M$ is a full {\it projective} $R$-lattice in $V$. 
  

An important property of maximal $R$-orders in $A$ is the
following. It plays an important role in the integral theory which
generalizes the ideal theory for Dedekind domains. 

\begin{thm}
  Every maximal $R$-order $\Lambda$ in $A$ is hereditary.
\end{thm}
\begin{proof}
  This is a consequence of Theorems~\ref{55} and \ref{48} (1).
\end{proof}\ 

The reader is referred to the last chapter of \cite{reiner:mo} for the
explicit description of global hereditary $R$-orders in $A$, 
which is beyond the scope of this Note.

\section{Maximal orders and abelian varieties}
\label{sec:06}

In this section we give a proof of Theorem~\ref{ab}. Theorem~\ref{ab}
follows from a more general statement (Theorem~\ref{65}) 
where the ring $\End_k(A_0)$ is replaced
by any subring $\calO$ in it and $\calO'$ by any order of 
$\calO\otimes \Q$ containing $\calO$. 
 
We are grateful to the referee for his/her kind suggestion of using
Serre's tensor product construction, which improves our earlier result
(Proposition~\ref{66}).
The construction is explained in \cite[1.6 and 4.2]{chai-conrad-oort}.

\subsection{A construction of Serre and properties}
\label{sec:61}

Let $A$ be an abelian variety over a field $k$. Let $\calO\subset
\End_k(A)$ be any subring, not necessarily be commutative. Note that
$\calO$ is finite and free as a $\Z$-module, so it is both left and right
Noetherian as a ring. 

Let $M$ be a finite right $\calO$-module. Consider the functor
$\calT$ from the category of $k$-schemes to the category of abelian
groups defined by $\calT(S):=M\otimes_\calO A(S)$ for any $k$-scheme $S$. 

\begin{lemma} \label{61} Notations as above.
The fppf sheaf associated to the group functor $\calT$ is representable
by an abelian variety $M\otimes_\calO A$ over $k$ 
\end{lemma}
\begin{proof}
 This is \cite[Proposition 1.6.4.3]{chai-conrad-oort} (the assumption
 that $\calO$ is commutative there is superfluous). We provide the
proof for the reader's convenience. Choose a
finite presentation of $M$ as $\calO$-modules: 
\[ 
\begin{CD}
\calO^r @>{\alpha}>> \calO^s @>>> M @>>> 0.   
\end{CD} \]
Note that $M$ is the quotient of the {\it abelian} group $\calO^r$ by
its subgroup $\alpha(\calO^r)$. 
Tensoring $\otimes_\calO A$, we obtain a morphism $\alpha_A:A^r\to
A^s$ of abelian varieties over $k$ and a short exact sequence of 
abelian groups
\[ 
\begin{CD}
A^r(S) @>{\alpha_S}>> A^s(S) @ >>> \calT(S) @>>> 0,   
\end{CD} \]
for any $k$-scheme $S$. 
The abelian group $\calT(S)$ is equal to the cokernel of the map $\alpha_S$.
On the other hand, the quotient abelian variety
$C:=A^s/\alpha(A^s)$ represents the cokernel of $\alpha_A$ as a
fppf abelian sheaf over $k$. This shows the representability. \qed 
\end{proof}

We examine some basic properties of this construction.
  
\begin{lemma} \label{62} \
  \begin{enumerate}
\item  Let $M_1\to M_2\to M_3\to 0$ be an exact sequence of
  finite right $\calO$-modules. Then the associated morphisms of
  abelian varieties over $k$
  \begin{equation}
    \label{eq:61}
   M_1\otimes_\calO A\to M_2\otimes_\calO A\to M_3\otimes_\calO A\to 0  
  \end{equation}
form an exact sequence.
\item If $M$ is a $\Z$-torsion $\calO$-module, then the abelian
  variety $M\otimes_\calO A$ is zero. Therefore, the map $M\to
  M/M_{\rm tors}$ of $\calO$-modules induces an isomorphism
\[ M\otimes_\calO A\simeq (M/M_{\rm tors})\otimes_\calO A \]
  of abelian varieties  
  over $k$, where.
 \[ M_{\rm tors}:= \{x\in M\,| \, nx=0 
  \mbox{\ for some $n\neq 0 \in \Z$} \}    \]
  is the $\Z$-torsion $\calO$-submodule of $M$. 
\item Let $\alpha:M_1\to M_2$ be a map of finite right $\calO$-modules. Then
  the induced morphism $\alpha_A:  M_1\otimes_\calO A\to
  M_2\otimes_\calO A$ is an isogeny if the map
  $\alpha_\Q:M_1\otimes \Q\to M_2\otimes \Q$ is an isomorphism.
\end{enumerate}
\end{lemma}
\begin{proof}
(1) First, the sequence of abelian
  groups
\[ M_1\otimes_\calO A(S) \to M_2\otimes_\calO A(S)\to M_3\otimes_\calO
  A(S)\to 0 \]
is exact for any $k$-scheme $S$. Since the fppf sheafification is the
inductive limit of the equalizers of all fppf covers and the inductive
limit is an exact functor, the sequence of the 
sheafifications of $\calT_i(S)=M_i\otimes_\calO A(S)$ is exact. 
That is, the sequence
(\ref{eq:61}) is exact  
as fppf abelian sheaves over $k$. 

(2) The natural morphism $M\otimes_\Z A\to
    M\otimes_\calO A$ is faithfully flat. Therefore, it suffices to
    show the case where $\calO$ is $\Z$ and we can even assume 
    that $M=\Z/n\Z$ because any finite
    abelian group is a finite product of finite cyclic groups. 
    Then we get an exact
    sequence of abelian varieties (from $n:\Z\to \Z$ with cokernel $\Z/n\Z$)
\[ A\stackrel{n}{\to} A\to \Z/n\otimes_\Z A\to 0. \] 
It follows that the abelian variety $\Z/n \otimes_\Z A$ is zero. 

(3) Using (2) we may assume that $M_1$ and $M_2$ are free
$\Z$-modules. If $\alpha_\Q$ is an isomorphism, then $\alpha$ is
injective with finite cokernel. Then $\alpha_A$ is an isogeny by
Proposition 1.6.4.3 of \cite{chai-conrad-oort}. \qed

\end{proof}

\begin{eg} The converse of Lemma~\ref{62} (2) does not hold. That is,
  there may be a map $\alpha$ such that the map $\alpha_\Q$ is not
  isomorphic but  
  the morphism $\alpha_A$ can be an isogeny. We give an example. Let
  $E$ be an 
  elliptic curve with $\End(E)=\Z$. Put $A:=E^{\oplus 2}$ and 
  $R:=\End(A)=\Mat_2(\Z)$. Let
\[ \calO:=\left \{
\begin{pmatrix}
  a & b \\ 0 & c
\end{pmatrix}\Big| \,a,b,c\in \Z\right \}. \]
Put $M_2=\Z^{\oplus 2}=\Z e_1+\Z e_2$, the free module of row
$\Z$-vectors endowed with the natural right $\calO$-module. 
Let $M_1:=\Z e_2$ be the invariant $\calO$-submodule of $M_2$ and
$\alpha:M_1\to M_2$ be the inclusion map. Let $M_3=\Z$ be the cokernel
of $\alpha$ and $\beta:M_2\to M_3$, $(a,b)\mapsto a$, be the natural
projection. The induced action of $\calO$ on $M_3$ is given by 
$1\cdot 
\begin{pmatrix}
  a & b \\ 0 & c
\end{pmatrix}= a.$
We have 
\[ M_1=e_2 \calO=\calO/I_1, \quad M_2=e_1 \calO=\calO/I_2, \quad M_3=1
\calO=\calO/I_3, \]
where
\[ I_1=\left \{\begin{pmatrix}
  a & b \\ 0 & 0
\end{pmatrix}\right \}, \quad I_2=\left \{\begin{pmatrix}
  0 & 0 \\ 0 & c
\end{pmatrix}\right \}, \quad I_3=\left \{\begin{pmatrix}
  0 & b \\ 0 & c
\end{pmatrix}\right \}. \]
Now $M_i\otimes_\calO A\simeq (M_i\otimes_\calO R) \otimes_R A$ and
$M_i\otimes_\calO R\simeq R/I_iR$. We easily see that
\[ I_1R=\left \{\begin{pmatrix}
  a & b \\ 0 & 0
\end{pmatrix}\right \}, 
\quad I_2R=\left \{\begin{pmatrix}
  0 & 0 \\ b' & c
\end{pmatrix}\right \}, \quad I_3R=\left \{\begin{pmatrix}
  a & b \\ b' & c
\end{pmatrix}\right \}. \]    
Therefore, $M_1\otimes_\calO A\simeq E$, $M_2\otimes_\calO A\simeq E$
and $M_3\otimes_\calO A=0$ and the morphism $\alpha_A$ is an isogeny.

In this example we see how to compute $M\otimes_\calO A$ when $M$ is
monogenetic. It seems that if $M_2$ is indecomposable and $M_1\subset
M_2$ is a nonzero $\calO$-submodule, then the morphism $\alpha_A$ 
induced from
the inclusion $\alpha$ is an isogeny. 
\end{eg}



\subsection{Relations with Tate and \dieu modules}
\label{sec:62}
Let $A$, $\calO$ and $M$ be as above. For any prime $\ell\neq
\char(k)$, denote by $T_\ell(A)$ the $\ell$-adic Tate module of $A$
viewed as a $\Gal(k^s/k)$-module, where $k^s$ is a separable closure
of $k$. If $\char(k)=p>0$ and $k$ is perfect, then denote by $\bfM(A)$
the covariant \dieu module of $A$.  

\begin{prop}\label{64}
  Let $A$, $\calO$ and $M$ be as in \S~\ref{sec:61}. Then there exist
  surjective maps 
\[ \xi_{\ell,M}: M\otimes_\calO T_\ell(A)\to
  T_\ell(M\otimes_\calO A) \] 
with finite kernel for any prime $\ell\neq \char(k)$, and a surjective map
\[ \xi_{p,M}: M\otimes_\calO \bfM(A)\to
  \bfM(M\otimes_\calO A) \]
with finite length kernel if $\char(k)=p>0$ and $k$ is
perfect. Moreover, if $\alpha:M_1\to M_2$ 
is a map of finite $\calO$-modules, then the following diagram for
the associated Tate modules for any prime $\ell\neq \char(k)$
\begin{equation}
  \label{eq:62}
\begin{CD}
  M_1\otimes T_\ell(A) @>{\alpha\otimes 1}>> M_2 \otimes T_\ell(A) \\
  @V{\xi_{\ell,M_1}}VV                        @V{\xi_{\ell,M_2}}VV \\
  T_\ell(M_1\otimes A) @>{T_\ell(\alpha_A)}>> T_\ell(M_2\otimes A) 
\end{CD}  
\end{equation}
(resp. for the associated \dieu modules) commutes.
\end{prop}
\begin{proof}
  Choose a finite presentation of $M$ as $\calO$-modules:
  \begin{equation}
    \label{eq:63}
\begin{CD}
\calO^r @>{\alpha}>> \calO^s @>>> M @>>> 0   
\end{CD}    
  \end{equation}
and get a morphism $\alpha_A: A^r\to A^s$ of abelian varieties over
$k$. Let $B$ be the image abelian variety of $\alpha_A$. We have a
short exact sequence of abelian varieties over $k$:
\[ 0\to B \to A^s \to M\otimes_\calO A\to 0. \]
This gives rise to a short exact sequence of Tate modules 
\[  0\to T_\ell(B) \to T_\ell(A^s) \to T_\ell(M\otimes_\calO A)\to 0 \]
and \dieu modules
\[ 0\to \bfM(B) \to \bfM(A^s) \to \bfM(M\otimes_\calO A)\to 0. \] 
On the other hand, tensoring the exact sequence (\ref{eq:63}) over
the Tate module $T_\ell(A)$ and the \dieu module $\bfM(A)$,
respectively, 
we get exact sequences
\[ \begin{CD}
T_\ell(A)^{\oplus r} @>{T_\ell(\alpha)}>> T_\ell(A)^{\oplus s} @>>> 
M\otimes_\calO T_\ell(A) @>>> 0,   
\end{CD} \] 
and 
\[ \begin{CD}
\bfM(A)^{\oplus r} @>{\bfM(\alpha)}>> \bfM(A)^{\oplus s} @>>> 
M\otimes_\calO \bfM(A) @>>> 0.   
\end{CD} \]
This gives a surjective map $\xi_{\ell,M}: M\otimes T_\ell(A)\to
T_\ell(M\otimes_\calO A)$ and a surjective map $\xi_{p,M}: M\otimes
\bfM(A)\to \bfM(M\otimes_\calO A)$. The kernel of the map $\xi_{\ell,M}$
(resp. $\xi_{p,M}$) is the cokernel of the map $T_\ell(\alpha):
T_\ell(A)^{\oplus r}\to T_\ell(B)$ (resp. $\bfM(\alpha):
\bfM(A)^{\oplus r}\to \bfM(B)$). This proves the first part of the
proposition.  

The natural map $\xi_{\ell,M}$ induces an natural isomorphism 
\[ \xi_{\ell,M}:M\otimes_\calO T_\ell(A)/({\rm torsion})\simeq
T_\ell(M\otimes_\calO A).\]
 From this it follows that the diagram (\ref{eq:62}) commutes. The
proof of the assertion for \dieu modules is the same. \qed 
\end{proof}

\subsection{Computation of $M\otimes_\calO A$ up to isogeny.}
\label{sec:63}

Let $[A]$ denote the isogeny class of an abelian variety $A$ over a
field $k$. Let $C$ be a $\Q$-subalgebra of the semi-simple algebra 
$\End^0_k(A)$ and let $V$ be a finite right $C$-module. It is not hard
to see from Lemma~\ref{62} that the isogeny class $[M\otimes_\calO A]$,
for a $\Z$-order $\calO$ in $C$ contained in $\End_k(A)$ and 
an $\calO$-lattice $M$, does not depend on the choice of $\calO$ and
$M$ (and also the choice of $A$ in $[A]$). 
We denote this isogeny class by $V\otimes_C [A]$.

Write $[A]=[\prod_{i=1}^r B_i^{n_i}]$ into a finite product of 
isotypic components, where
each $B_i$ is a $k$-simple abelian variety and $B_i$ is not isogenous to
$B_j$ for $i\neq j$. 
The endomorphism algebra $E:=\End^0(A)\simeq \prod_{i=1}^r
\Mat_{n_i}(D_i)$ decomposes into the product of its simple factors,
where $D_i$ is the endomorphism algebra of $B_i$. Observe that
\begin{equation}
  \label{eq:64}
  V\otimes_C [A]= V\otimes_C  E\otimes_E [A].
\end{equation}
If we write 
\begin{equation}
  \label{eq:65}
  V\otimes_C  E\simeq \bigoplus_{i=1}^r I_i^{m_i}
\end{equation}
as $E$-modules, where $I_i$ is a minimal non-zero ideal of
$\Mat_{n_i}(D_i)$, then  
\[ V\otimes_C [A]\simeq \left (\bigoplus_{i=1}^r I_i^{\oplus m_i}
\right) \otimes_E
[A]=\prod_{i=1}^r I_i^{\oplus m_i}\otimes_{\Mat_{n_i}(D_i)}
  [B_i^{n_i}]\simeq \prod_{i=1}^r [B_i^{m_i}].  \]

Therefore, the computation of the abelian variety 
$M\otimes_\calO A$ up to isogeny is  reduced to
the simple algebra problem (\ref{eq:65}) of decomposing the module
$V\otimes_C E$ into simple $E$-modules.
 
The dimension of $M\otimes_\calO A$ is given by the formula:
\begin{equation}
  \label{eq:66}
  \dim M\otimes_\calO A=\sum_{i=1}^r m_i \dim B_i, 
\end{equation}
where $m_i$ are the integers in (\ref{eq:65}).

\subsection{On minimal isogenies for abelian varieties.}
\label{sec:64}
The main result of this section is the following theorem.

\begin{thm}\label{65}
  Let $A_0$ be an abelian variety over $k$, $\calO$ a subring of
  $\End_k(A_0)$, and $\calO'$ a $\Z$-order of $\calO\otimes_\Z \Q$
  containing 
  $\calO$. Then the isogeny $\iota: A_0\to \calO'\otimes_\calO A_0$ 
  satisfies the following property: for any pair $(\varphi_1,A_1)$
  where $\varphi_1:A_0\to A_1$ is an isogeny of abelian varieties over
  $k$ such that with the identification $\End^0(A_0)=\End^0(A_1)$ by
  $\varphi_1$ one has $\calO'\subset \End(A_1)$, then there is a
  unique isogeny $\alpha:\calO'\otimes_\calO A_0\to A_1$ over $k$ such
  that $\varphi_1=\alpha\circ \iota$. 
\end{thm}

Theorem~\ref{ab} is the special case of Theorem~\ref{65} where
$\calO=\End_k(A_0)$ and $\calO'$ is a maximal order containing $\calO$.
We first prove a weaker statement.

\begin{prop}\label{66}
  Let $A_0$, $\calO$, $\calO'$ be as in Theorem~\ref{65}.  
  Then there exist a finite purely inseparable extension $k'$ of $k$, 
  an abelian variety $A'$ over $k'$, and an isogeny $\varphi:A_0\otimes
  k'\to A'$ over $k'$ such that with the identification
  $\End^0_{k'}(A\otimes k')=\End_{k'}^0(A')$ one has 
   $\calO'\subset \End_{k'}(A')$. Moreover, the
  isogeny $\varphi$ can be chosen to be minimal with respect to
  $\calO'$ in the following sense: if $(k_1, A_1, \varphi_1)$ is
  another triple with the property $\calO'\subset \End_{k_1}(A_1)$, 
  then there are a 
  finite purely inseparable extension $k''$ of $k$ containing both 
  $k'$ and $k_1$ 
  and a unique isogeny 
\[ \alpha: A'\otimes_{k'} k''\to A_1\otimes_{k_1} k''  \]
  such that $\varphi_{1, k''}=\alpha\circ \varphi_{k''}$. 
\end{prop}
\begin{proof}
Replacing $k$ by a  finitely generated subfield over its prime field
and replacing $A_0$ by a model of $A_0$ defined 
over this subfield whose endomorphism ring
is equal to $\End_k(A_0)$, we may assume that the ground field $k$
is finitely generated over its prime field. Put
$\calG:=\Gal(k^{s}/k)$, where $k^s$ is a separable closure of
$k$. For any prime $\ell\neq \char k$, let $T_\ell:=T_\ell(A_0)$ be
the associated  Tate module of $A_0$, and
\[ \rho_{\ell}: \calG\to \Aut(T_\ell) \]
be the associated Galois representation. Let $G_\ell\subset
\Aut(T_\ell)$ 
be the image of the map
$\rho_\ell$, and write $\calO_\ell:=\calO\otimes_{\Z} \Z_\ell$
and $\calO'_\ell:=\calO' \otimes_{\Z} \Z_\ell$, respectively. 
By the theorem
of Tate, Zarhin and Faltings \cite{tate:eav,zarhin:end,faltings:end} on
homomorphisms of abelian varieties, we have 
\[ \calO_\ell\subset \End_k(A_0)\otimes \Z_\ell =C_{\End(T_\ell)}
G_\ell, \] 
the centralizer of $G_\ell$ in $\End(T_\ell)$. Let $T'_\ell$ be the
$\calO'_\ell$-submodule in $V_\ell:=T_\ell\otimes_{\Z_\ell} \Q_\ell$
generated by $T_\ell$. Since the action of $\calO_\ell$ on $T_\ell$ 
commutes with that of $G_\ell$, the lattice $T'_\ell$ is stable under
the $G_\ell$-action. As $\calO'_\ell=\calO_\ell$ for almost all primes
$\ell$, we have $T'_\ell=T_\ell$ for such primes $\ell$. If $\char k
=0$, then by a theorem of Tate, there are an abelian variety $A'$ over
$k$ and an isogeny $\varphi:A_0\to A'$ over $k$ such that the image of
$T_\ell(A')$ in $V_\ell$ by $\varphi$ is equal to $T'_\ell$ for all
primes $\ell$. We have $\calO'_\ell\subset
\End_{G_\ell}(T'_\ell)=\End_k(A')\otimes \Z_\ell$, and hence
$\calO'\subset \End_k(A')$.  

Suppose $\char k=p>0$. Let $k^{\rm pf}$ be the perfect closure of
$k$. Let $M_0:=\bfM(A_0\otimes_k k^{\rm pf})$ be the covariant \dieu module
of $A_0\otimes_k k^{\rm pf}$, on
which the ring $\calO_p:=\calO\otimes_{\Z} \Z_p$ acts. 
Let $W$ be the ring of Witt vectors over $k^{\rm
  pf}$ and let $B(k^{\rm pf})$ be its fraction field.
Let $M'$ be the $\calO'_p \otimes_{\Z_p} W$-submodule of 
$M_0\otimes_W B(k^{\rm  pf})$ generated by $M_0$. 
Since the action of $\calO_p$ on $M_0$
commutes with the Frobenius map $F$, the submodule $M'$ is 
a \dieu module containing $M_0$. By a theorem of Tate, 
there are an abelian variety $A$ over
$k^{\rm pf}$ and an isogeny $\varphi_{k^{\rm pf}}:
A_0\otimes_k k^{\rm pf} \to A$
over $k^{\rm pf}$ such that the image of
the Tate module $T_\ell(A)$ in  $V_\ell$ and the \dieu module $\bfM(A)$
in $M_0\otimes_W B(k^{\rm pf})$) by the isogeny $\varphi$ 
is equal to $T'_\ell$ for all
primes $\ell\neq p$ and equal to $M'$ at the prime $p$, respectively. 
Similarly, we show $\calO'\subset \End_{k^{\rm
    pf}}(A)$. Since $\ker \varphi_{k^{\rm pf}}$ is of finite
type, there is a model $A'$ of $A$ over a finite
extension $k'$ of $k$ in $k^{\rm pf}$ so that the isogeny
$\varphi_{k^{\rm pf}}$ is defined over $k'$. 

If we have another triple $(k_1,A_1,\varphi_1)$ with
$\calO'\subset \End_{k_1}(A)$, then the Tate module $T_\ell(A_1)$
viewed as a lattice in $V_\ell$ by the isogeny
$\varphi_1$ is an $\calO'_\ell [G_\ell]$-stable lattice, 
and hence it contains $T'_\ell$. Similarly, one shows that the \dieu
module $M_1$ of $A_1\otimes k^{\rm pf}$ as a \dieu sublattice in
$M_0\otimes B(k^{\rm pf})$ by the isogeny
$\varphi_1$ contains $M_0$. 
Therefore, there is an isogeny $\alpha:A'\otimes_{k'}
k^{\rm pf}\to A_1\otimes_{k_1} k^{\rm pf}$ such that
$\varphi_{1,k^{\rm pf}}=\alpha \circ \varphi_{k^{\rm pf}}$. Clearly,
the morphism $\alpha$ is defined over some finite extension of $k$ in
$k^{\rm pf}$ containing $k'$ and $k_1$. This proves the proposition. \qed
\end{proof}

\begin{lemma}[Chow]\label{67}
Let $A$ and $B$ are two abelian varieties over a field $k$, and let
$K/k$ be a primary field extension (i.e. $k$ is separably
algebraically closed in $K$). Then the natural map 
\[ \Hom_k(A, B) \to \Hom_K(A\otimes_k K, B\otimes_k K) \]
is bijective.
\end{lemma}
\begin{proof}
  See \cite[Lemma 1.2.1.2]{chai-conrad-oort}.\qed 
\end{proof}

We are ready to prove Theorem~\ref{65}. We show that 
the isogeny $\iota:A_0\to \calO'\otimes_\calO A_0$ satisfies the
property of Proposition~\ref{66}. Put $A'_0:=\calO'\otimes_\calO A_0$.
By Proposition~\ref{64}, the Tate module
$T_\ell(A'_0)$ of  $A_0'$ in $V_\ell=T_\ell(A_0)\otimes \Q_\ell$
through the isogeny $\iota$ is equal to $T_\ell'$ (in the proof of
Proposition~\ref{66}), and its \dieu
module $\bfM(A_0'\otimes k^{\rm pf})$ of the
abelian variety $A_0'\otimes k^{\rm pf}$
in $M_0 \otimes_W B(k^{\rm  pf})$ through $\iota$ is equal to
$M'$ (in the proof of
Proposition~\ref{66}), where $M_0=\bfM(A_0)$. 
Therefore, there is an isomorphism $A_0'\otimes_k k^{\rm pf}\simeq
A'\otimes_{k'} k^{\rm pf}$ which transforms $\iota$ to $\varphi$. So
the isogeny $\iota$ satisfies the property of Proposition~\ref{66}. 

Now by Lemma~\ref{67}, the morphism $\alpha$ in Proposition~\ref{66},
which is a priori defined over some finite purely inseparable
extension of $k$, is defined over $k$. The proof of Theorem~\ref{65}
is complete. 





\begin{thank}
  The author learned a lot from the books 
  \cite{matsumura:ca80} and \cite{reiner:mo} while working on both maximal
  orders and Nagata rings and preparing the present Note. 
  He is grateful to Matsumura and Reiner for
  their tremendous efforts on writing the great books. He is also
  grateful to Tse-Chung Yang, and especially to C.-L.~Chai for very helpful
  discussions. Main part of
  the manuscript was prepared while the author's stay at the
  Institute of Mathematical Sciences, Hong Kong Chinese University. He
  thanks the institute for kind hospitality and good working
  conditions. The research was partially supported by 
  the grants NSC 97-2115-M-001-015-MY3 and AS-99-CDA-M01. Finally, he
  is grateful to the referee for his/her kind suggestion of using
  Serre's construction which improves Theorem~\ref{ab}. 
\end{thank}

\end{document}